\documentclass[12pt]{article}

\usepackage[english]{babel}
\usepackage{graphicx}
\usepackage{latexsym}
\usepackage{amsfonts}
\usepackage{dsfont}
\usepackage{amscd}
\usepackage{amssymb}
\usepackage{amsmath}
\usepackage{mathtools}
\usepackage{mathrsfs}
\usepackage{amsthm}
\usepackage{tikz}

\input xy
\xyoption{all}

\newcommand{\p}{\varphi}

\newcommand{\bz}{\mathbb{Z}}
\newcommand{\F}{\mathbb{F}}

\newcommand{\s}{\mathcal{S}}
\newcommand{\Et}{\widehat E_{\text{tight}}}
\newcommand{\Eu}{\widehat E_{\infty}}
\newcommand{\Gt}{\mathcal{G}_{\text{tight}}}
\newcommand{\G}{\mathcal{G}}

\theoremstyle{definition}
\newtheorem{theo}{Theorem}[section]

\newtheorem{ex}[theo]{Example}
\newtheorem{lem}[theo]{Lemma}
\newtheorem{defn}[theo]{Definition}

\begin{document}

\title{Self-Similar Graph C*-Algebras and Partial Crossed Products}
\author{Ruy Exel\thanks{Partially supported by CNPq (Brazil).} \hspace{1cm} Charles Starling\thanks{Supported by CNPq (Brazil).}}

\date{}
\maketitle
\begin{abstract}
In a recent paper, Pardo and the first named author introduced a class of C*-algebras which which are constructed from an action of a group on a graph. This class was shown to  include many C*-algebras of interest, including all Kirchberg algebras in the UCT class. In this paper, we study the conditions under which these algebras can be realized as partial crossed products of commutative C*-algebras by groups. In addition, for any $n\geq 2$ we present a large class of groups such that for any group $H$ in this class, the Cuntz algebra $\mathcal{O}_n$ is isomorphic to a partial crossed product of a commutative C*-algebra by $H$.
\end{abstract}

\section{Introduction}

In this paper, we study C*-algebras arising from {\em self-similar graph actions}. A self-similar graph action consists of a group $G$, a directed graph $E$, a ``cocycle'' $\p: G \times E^1 \to G$, and a length-preserving action of $G$ on $E^*$, the set of finite paths through $E$. Furthermore, the action of $G$ on $E^*$ should be ``self-similar'', in the sense that
\[
g(e\alpha) = (ge)(\p(g, e)\alpha), \hspace{1cm} g\in G, e\in E^1, \alpha\in E^*.
\]
The C*-algebra associated to $(G, E, \p)$, denoted $\mathcal{O}_{G, E}$ is then the universal C*-algebra generated by a Cuntz-Krieger family for $E$ and a unitary representation of $G$, subject to relations given by the action. The main question this work addresses is the following: when can $\mathcal{O}_{G, E}$ be written as a partial crossed product of a commutative C*-algebra by a group?

Some of the most powerful tools to analyze the structure of a given C*-algebra become available when one is able to describe it as a crossed product.  Questions about representation theory, structure of ideals, simplicity, nuclearity, K-theory, KMS states, and many others may be answered, under suitable hypothesis, when the algebra under analysis is given the structure of a crossed product.  Unfortunately, not many
algebras of interest may be described as a crossed product relative to a global group action, but if we resort to partial actions, the chances of obtaining such a description improve manifold, at the same time that most of the above mentioned tools are still available for partial crossed products.

The motivation for defining the algebras $\mathcal{O}_{G, E}$ in \cite{EP13} was to generalize two interesting classes of C*-algebras. The first class is the C*-algebras that Nekrashevych associated to self-similar groups in \cite{Nek04} and \cite{Nek09}. In \cite{EP13}, it was shown that these algebras arise from self-similar graph actions $(G, E, \p)$ for which $E$ is a finite graph with only one vertex. In the work of Nekrashevych and in other work on self-similar groups, the map $\p$ takes the form of a {\em restriction}, $(g, x)\mapsto \left.g\right|_x$.

The second class generalized in \cite{EP13} are the C*-algebras constructed by Katsura \cite{Ka08} from pairs of integer matrices $A$ and $B$, denoted $\mathcal{O}_{A, B}$. The pair $(A, B)$ gives rise to a self-similar graph action $(\bz, E_A, \p)$, where $E_A$ is the graph whose incidence matrix is $A$ and the action of $\bz$ and the cocycle $\p$ are determined by the entries of $A$ and $B$. It is shown in \cite{EP13}, Example 3.4, that $\mathcal{O}_{A, B} \cong \mathcal{O}_{\bz, E_A}$. From \cite{Ka08}, it is a fact that every Kirchberg algebra in the UCT class arises as $\mathcal{O}_{A, B}$ for some $A$ and $B$, so the algebras we consider here constitute a large class. 

In \cite{EP13}, $\mathcal{O}_{G, E}$ is realized as the C*-algebra of the tight groupoid of a certain inverse semigroup $\s_{G, E}$, using the theory of such algebras from \cite{Ex08}. This groupoid is shown to be Hausdorff in the case where $(G, E, \p)$ is {\em pseudo-free}. In Theorem \ref{RFstronglyE*unitary}, we show that $(G, E, \p)$ is pseudo-free if and only if inverse semigroup $\s_{G, E}$ is {\em strongly E*-unitary}. We then use the results of \cite{SM11} to prove that when $(G, E, \p)$ is pseudo-free, $\mathcal{O}_{G, E}$ is isomorphic to a partial crossed product of a commutative C*-algebra by the universal group $U(\s_{G, E})$ of $\s_{G, E}$.

In our final section, we use our results to realize the Cuntz algebra $\mathcal{O}_n$ as partial crossed products $C(Y)\rtimes H$ for $Y$ homeomorphic to the Cantor set and $H$ the universal group of {\em any} $\s_{G, E}$ satisfying certain conditions. 

\section{Terminology, Notation, and Background}

Let $X$ be a finite set. We will denote by $X^*$ the set of all {\em words} $x_1x_2\cdots x_n$ where $n\geq 1$ and $x_1, x_2, \dots, x_n\in X$, together with the symbol $\varnothing$, which is called the {\em empty word}. For $\alpha = \alpha_1\alpha_2, \cdots\alpha_n\in X^*$, we let $|\alpha| = n$ and call this the {\em length} of $\alpha$ -- we take $|\varnothing| = 0$. We may concatenate two words in $X^*$: if $\alpha, \beta\in X^*$ with $\alpha = \alpha_1\alpha_2\cdots\alpha_n$ and $\beta =\beta_1\beta_2\cdots\beta_m$ then their concatenation is given by
\[
\alpha\beta = \alpha_1\alpha_2\cdots\alpha_n\beta_1\beta_2\cdots\beta_m.
\]
We also take, for any $\alpha\in X^*$,
\[
\alpha \varnothing = \varnothing\alpha = \alpha.
\]
This operation gives $X^*$ the structure of a semigroup with identity (or {\em monoid}) and is called the {\em free monoid} on $X$. We may also consider the space $\Sigma_X$ of {\em infinite} words in $X$. For an element $x\in \Sigma_X$ and $\alpha\in X^*$, the concatenation $\alpha x$ is again in $\Sigma_X$. For $\alpha\in X^*$, let
\[
\alpha \Sigma_X = \{\alpha x\in \Sigma_X \mid x\in \Sigma_X\}.
\]
Sets of this type are called {\em cylinder sets}, and they generate the product topology on $\Sigma_X$. In this topology, cylinder sets are both open and closed and $\Sigma_X$ is homeomorphic to the Cantor set.

A {\em directed graph} is a quadruple $E = (E^0, E^1, r, d)$ where $E^0$ and $E^1$ are sets and $r, d$ are functions from $E^1$ to $E^0$. The set $E^0$ is called the set of {\em vertices} of $E$ and $E^1$ is called the set of {\em edges} of $E$. A vertex $x\in E^0$ is said to be a {\em source} if $r^{-1}(x) = \emptyset$, and it is said to be a {\em sink} if $d^{-1}(x) =\emptyset$. A directed graph $E$ is called {\em finite} if $E^0$ and $E^1$ are finite sets. For $n\geq 2$ we let
\[
E^n = \{ x_1x_2\cdots x_n\mid x_i \in E^1, d(x_i) = r(x_{i+1}) \text{ for } 1\leq i \leq n-1\}
\]
and take
\[
E^* = \bigcup_{n = 0}^\infty E^n.
\]
The set $E^*$ is called the set of {\em finite paths} in $E$. We may concatenate two paths $\alpha$ and $\beta$ and obtain $\alpha\beta$ if $r(\beta) = d(\alpha)$, taking the convention that $d(x) = r(x) = x$ for all vertices $x$, and that $r(\alpha)\alpha = \alpha = \alpha d(\alpha)$. We will also consider the set of {\em infinite paths}
\[
\Sigma_E = \{e_1e_2\cdots \mid e_i\in E^1, s(e_i) = r(e_{i+1})\text{ for all }i \geq 1\}
\]
given the product topology. With an abuse of notation, for $\alpha \in E^*$ we will let
\[
\alpha\Sigma_E = \{\alpha x\mid x\in E^*, r(x) = d(\alpha)\}
\]
and call these cylinder sets as well -- sets of this type generate the topology on $\Sigma_E$. If $E$ has one vertex, we may refer to paths in $E$ as ``words''

A semigroup $\s$ is called an {\em inverse semigroup} if for each $s\in\s$ there exists a unique element $s^*\in\s$ such that $s^*ss^* = s^*$ and $ss^*s = s$. An element $e\in \s$ is called an {\em idempotent} if $e^2 = e$ -- the set of all such elements will be denoted $E(\s)$. It is true that, for all $s, t\in\s$ and $e, f\in E(\s)$ we have that $(st)^* = t^*s^*$, $(s^*)^* = s$, $e^* = e$, $ef = fe$ and $ef \in E(\s)$. For all $s\in \s$, the elements $ss^*$ and $s^*s$ are idempotents.

Recall that a {\em groupoid} is a set $\G$ together with a subset $\G^{(2)} \subset \G \times \G$, called the set of composable pairs, a product map $\G^{(2)} \to \G$ with $(\gamma, \eta)\mapsto \gamma\eta$, and an inverse map from $\G$ to $\G$ with $\gamma \mapsto \gamma^{-1}$ such that
\begin{enumerate}
\item $(\gamma^{-1})^{-1} = \gamma$ for all $\gamma\in \G$,
\item If $(\gamma, \eta), (\eta, \nu)\in \G^{(2)}$, then $(\gamma\eta,\nu), (\gamma, \eta\nu)\in \G^{(2)}$ and $(\gamma\eta)\nu = \gamma(\eta\nu)$,
\item $(\gamma, \gamma^{-1}), (\gamma^{-1},\gamma)\in \G^{(2)}$, and $\gamma^{-1}\gamma\eta = \eta$, $\xi\gamma\gamma^{-1}$ for all $\eta, \xi$ with $(\gamma, \eta), (\eta,\xi) \in \G^{(2)}$.
\end{enumerate}
Elements of the form $\gamma\gamma^{-1}$ are called {\em units}, and the set of all such elements is denoted $\G^{(0)}$ and is called the {\em unit space} of $\G$. The maps $r: \G\to \G^{(0)}$ and $d:\G\to \G^{(0)}$ defined by
\[
r(\gamma) = \gamma\gamma^{-1}, \hspace{1cm} d(\gamma) = \gamma^{-1}\gamma
\]
are called the {\em range} and {\em source} maps respectively. We note that $(\gamma, \eta)\in \G^{(2)}$ if and only if
$r(\eta) = d(\gamma)$. A {\em topological groupoid} is a groupoid which is a topological space for which the inverse and
product maps are continuous (where $\G^{(2)}$ is given the product topology inherited from $\G\times\G$). An {\em
isomorphism} of topological groupoids is a homeomorphism which preserves the groupoid operations.  A topological
groupoid $\G$ is called {\em \'etale} if it is locally compact, second countable, and the maps $r$ and $d$ are local homeomorphisms. We note that these imply that $\G^{(0)}$ is open in $\G$ and that for all $x\in \G^{(0)}$ the spaces
\[
\G^x := r^{-1}(x), \hspace{1cm} \G_x := d^{-1}(x)
\]
are discrete.

An open set $S\subset \G$ of a topological groupoid is called a {\em slice} if the restrictions of $r$ and $d$ to $S$ are both injective. In an \'etale groupoid $\G$, the collection of slices forms a basis for the topology of $\G$, cf. \cite{Ex08}, Proposition 3.5.

\'Etale groupoids can be constructed from actions on topological spaces. Recall that an {\em action} of an inverse semigroup $\s$ on a space $X$ is a pair
\[
(\{D_e\}_{e\in E(\s)}, \{\theta_s\}_{s\in \s})
\]
such that each $D_e \subset X$ is an open set, the union of the $D_e$ coincides with $X$, and the maps
\[
\theta_s: D_{s^*s}\to D_{ss^*}
\]
are continuous bijections which satisfy $\theta_s\circ\theta_t = \theta_{st}$, where the composition is on the largest domain possible. Given an action $\theta$ of an inverse semigroup $\s$ on a space $X$, one may form an \'etale groupoid $\G(\s, X, \theta)$, called the {\em groupoid of germs}. As a set,
\[
\G(\s, X, \theta) = \{ [s,x]\mid x\in D_{s^*s}\}
\]
where $[s,x]$ are equivalence classes of elements of $\s\times X$ such that $[s,x] = [t,y]$ iff $x=y$ and there exists
some idempotent $e\in E(\s)$ such that $x\in D_e$ and $se = te$. One may always assume that $s^*se = e$ and $t^*te = e$, and we note that this implies that $\theta_s$ and $\theta_t$ agree on the neighborhood $D_e$ of $x$. The unit space of $\G(\s, X, \theta)$ is identified with $X$, and
\[
[s,x]^{-1} = [s^*, \theta_s(x)], \hspace{0.5cm} r([s,x]) = \theta_s(x), \hspace{0.5cm} d([s,x]) = x, \hspace{0.5cm} [t, \theta_s(x)][s,x] = [ts, x].
\]

A related construction arises from a group $G$ acting on a topological space by partial homeomorphisms. Let $X$ be a set and let $G$ be a group. Recall that a {\em partial action} of $G$ on $X$ is a pair
\[
(\{D_g\}_{g\in G}, \{\theta_g\}_{g\in G})
\]
consisting of a collection $\{D_g\}_{g\in G}$ of subsets of $X$, and a collection $\{\theta_g\}_{g\in G}$ of functions
\[
\theta_g: D_{g^{-1}}\to D_g,
\]
such that $D_1 = X$, $\theta_1$ is the identity map, and $\theta_g\circ\theta_h \subset \theta_{gh}$. Here,
$\theta_g\circ\theta_h$ may be ambiguous, because the range of $\theta_h$ may not be contained in the domain
$\theta_g$. This function is then defined on the largest domain possible, that is, $\theta^{-1}_{h}(D_h\cap D_{g^{-1}})$. On this domain it is defined to be $\theta_g\circ \theta_h$. The notation $\theta_g\circ\theta_h \subset \theta_{gh}$ means that the function $\theta_{gh}$ extends the function $\theta_g\circ\theta_h$. With the above, one can show that each $\theta_g$ is a bijection on its domain and that $\theta^{-1}_g = \theta_{g^{-1}}$.

A partial topological dynamical system is a quadruple
\[
(X, G, \{D_g\}_{g\in G}, \{\theta_g\}_{g\in G})
\]
such that $X$ is a topological space, $G$ is a group, $(\{D_g\}_{g\in G}, \{\theta_g\}_{g\in G})$ is a partial action of $G$ on the set $X$ such that each $D_g$ is an open subset of $X$ and each $\theta_g$ is a homeomorphism.

Given a partial topological dynamical system $\theta = (X, G, \{D_g\}_{g\in G}, \{\theta_g\}_{g\in G})$ with $X$ a locally compact Hausdorff space and $G$ discrete, using a construction of Abadie \cite{Ab04} we may form the \'etale groupoid
\begin{equation}\label{abadiegroupoid}
G \ltimes_\theta X := \{ (g, x) \in G\times X\mid x\in X_{g^{-1}}\}
\end{equation}
with $(G \ltimes_\theta X)^{(2)} = \{ ((g, x),(h, y)) \mid \theta_h(y) = x\}$, $r(g, x) = \theta_g(x), d(g, x) = x$ and $(g, x)^{-1} = (g^{-1}, \theta_{g}(x))$. If $y = \theta_g(x)$, then we have $(g, x)(h,y) = (gh, y)$.

There is a procedure for producing a C*-algebra from an \'etale groupoid $\G$, given by Renault in \cite{R80}. One considers the linear space $C_c(\G)$ of compactly supported complex functions on $\G$ equipped with product and involution given by
\[
(fg)(\gamma) = \sum_{r(\eta) = s(\gamma)} f(\gamma\eta)g(\eta^{-1})
\]
\[
f^*(\gamma) = \overline{f(\gamma^{-1})}.
\]
These give $C_c(\G)$ the structure of a complex $*$-algebra. The C*-algebra of $\G$ is then defined to be the completion
of this $*$-algebra in a certain norm, and is denoted $C^*(\G)$. It is a fact that if $\G$ and $\mathcal{H}$ are
isomorphic \'etale groupoids, then their C*-algebras are isomorphic. 

A partial C*-dynamical system is a quadruple
\[
(A, G, \{D_g\}_{g\in G}, \{\theta_g\}_{g\in G})
\]
such that $A$ is a C*-algebra, $G$ is a group, $(\{D_g\}_{g\in G}, \{\theta_g\}_{g\in G})$ is a partial action of $G$ on the set $A$ such that each $D_g$ is a closed two-sided ideal of $A$ and each $\theta_g$ is a $*$-isomorphism. If we are given a partial C*-dynamical system as above such that $G$ is a discrete group, then we can form the partial crossed product by first considering the $*$-algebra of formal finite linear combinations
\[
\sum_{g\in G}a_g\delta_g
\]
such that $a_g\in D_g$ for each $g\in G$. Here the symbols $\delta_g$ have no meaning other than placeholders.\footnote{One may consider instead the linear space of finitely supported functions $f$ from $G$ into $A$ such that $f(g)\in D_g$, and in this case, the element $a\delta_g$ is identified with the function which takes the value $a$ on $g$ and $0$ elsewhere.} Addition and scalar multiplication are defined in the obvious way, and multiplication is determined by the rule
\[
(a\delta_g)(b\delta_h) = \theta_g(\theta_{g^{-1}}(a)b)\delta_{gh}
\]
while the $*$ is determined by
\[
(a\delta_g)^* = \theta_{g^{-1}}(a^*)\delta_{g^{-1}}
\]
This multiplication is not always associative, but for the cases we consider here it is.

A partial topological dynamical system $(X, G, \{D_g\}_{g\in G}, \{\theta_g\}_{g\in G})$ with $X$ locally compact and Hausdorff gives rise to a partial C*-dynamical system $(C_0(X), G, \{C_0(D_g)\}_{g\in G}, \{\widetilde\theta_g\}_{g\in G})$ with
\[
\widetilde\theta_g: C_0(D_{g^{-1}}) \to C_0(D_{g})
\]
\[
\left.\widetilde\theta_g(f)\right|_x = f(\theta_g^{-1}(x))
\]
and $C_0(D_g)$ is understood to denote all the functions $f\in C_0(X)$ which vanish outside of $D_g$. It is a fact that in this situation, the resulting partial action crossed product is isomorphic to the C*-algebra of the groupoid $G \ltimes_\theta X$. An important special case is the situation where $X$ and each $D_g$ is compact -- in this case $C_0(X)$ becomes $C(X)$ and each $C_0(D_g)$ becomes $C(D_g)$.

 \section{Inverse Semigroups from Graph Actions}
Let $E = (E^0, E^1, r, d)$ be a finite directed graph. An {\em automorphism} of $E$ is a bijective map
\[
h: E^0\cup E^1 \to E^0\cup E^1
\]
such that $h(E^i)\subset E^i$ for $i = 0,1$ and also such that $h\circ d = d\circ h$ and $h\circ r = r\circ h$. An {\em action} of a group $G$ on a graph $E$ is a group homomorphism from $G$ to the group of automorphisms of $E$.

Suppose that $G$ acts on a set $X$. A {\em one-cocycle} for the action of $G$ on $X$ is a map
\[
\p: G \times X \to G
\]
such that
\[
\p(gh, x) = \p(g, hx)\p(h,x)
\]
for all $g, h\in G$ and $x\in X$. Setting $g=h=1$ into the above we get that $\p(1, x) = 1$ for each $x$.

We now assume that $E$ is a finite directed graph with no sources or sinks, $G$ is a countable discrete group, and that we have a homomorphism from $G$ to the group of automorphisms of $E$. We denote the image of a group element $g$ under this homomorphism simply as $g$. We also assume that we have a one-cocycle $\p : G\times E^1 \to G$ for the restriction of our action to $E^1$, which also satisfies
\[
\p(g, e)x = gx, \hspace{1cm}\text{ for all } g\in G, \hspace{0.5cm} e\in E^1, \hspace{0.5cm} x\in E^0
\]
In \cite{EP13}, they show that the action of $G$ and the cocycle $\p$ extend to $E^*$ in a natural way. This induced action preserves lengths. Furthermore, for every $g, h\in G$, for every $x\in E^0$ and for every $\alpha$ and $\beta$ in $E^*$ such that $d(\alpha) = r(\beta)$ we have

\vspace{-0.3cm}\begin{tabular}{p{7cm} p{7cm}}
\begin{enumerate}\addtolength{\itemsep}{-0.5\baselineskip}
\item[(E1)] $(gh)\alpha = g(h\alpha)$
\item[(E2)] $\p(gh, \alpha) = \p(g, h\alpha)\p(h, \alpha)$
\item[(E3)] $\p(g, x) = g$
\item[(E4)] $r(g\alpha) = gr(\alpha)$
\end{enumerate}
&
\begin{enumerate}\addtolength{\itemsep}{-0.5\baselineskip}
\item[(E5)] $d(g\alpha) = gd(\alpha)$
\item[(E6)] $ \p(g, \alpha)x = gx$
\item[(E7)] $g(\alpha\beta) = (g\alpha)\p(g,\alpha)\beta$
\item[(E8)] $\p(g, \alpha\beta) = \p(\p(g, \alpha),\beta)$.
\end{enumerate}
\end{tabular}

\vspace{-0.3cm}\noindent The triple $(G, E, \p)$ is called a {\em self-similar graph action}.

Given a self-similar graph action $(G, E, \p)$, we construct an action of $G$ and a cocycle on the graph obtained from $E$ by collapsing all the vertices to a single vertex. To be more precise, consider the directed graph $\widetilde E :=(\{\varnothing\}, E^1, r, d)$, that is, the directed graph with one vertex whose edge set is equal to the edge set of $E$. Then the set of paths $\widetilde E^*$ is just the free monoid on $E^1$, with identity equal to the empty word $\varnothing = r(e) = d(e)$ for all $e\in E^1$.

It is clear that our given action of $G$ on $E$ induces an action of $G$ on $\widetilde E$, and that here we have $g\varnothing = \varnothing$ for all $g\in G$. Furthermore, the cocycle $\p$ is defined on $G \times E^1$, so it is also a cocycle for the induced action on $\widetilde E$. For clarity, we denote the induced cocycle on $G \times \widetilde E^*$ by $\widetilde\p$. By the above list, we have for every $w, v\in \widetilde E^*$ and $g, h \in G$:

\vspace{-0.3cm}\begin{tabular}{p{7cm} p{7cm}}
\begin{enumerate}\addtolength{\itemsep}{-0.5\baselineskip}
\item[(SS1)] $1w = w$
\item[(SS2)] $(gh)w = g(hw)$
\item[(SS3)] $g\varnothing = \varnothing$
\item[(SS4)] $g(vw) = (gv)\widetilde\p(g, v)w$
\end{enumerate}
&
\begin{enumerate}\addtolength{\itemsep}{-0.5\baselineskip}
\item[(SS5)] $\widetilde\p(g, \varnothing) = g$
\item[(SS6)] $\widetilde\p(g, vw) = \widetilde\p(\widetilde\p(g, v), w)$
\item[(SS7)] $ \widetilde\p(1, w) = 1$
\item[(SS8)] $\widetilde\p(gh, w) = \widetilde\p(g, hw)\widetilde\p(h, w)$
\end{enumerate}
\end{tabular}

\vspace{-0.3cm}\noindent These properties mean that the pair $(G, \widetilde E^1)$ is a {\em self-similar action} in the sense of Lawson (\cite{La08}, Section 3), which he proves is equivalent to $(G, \widetilde E^1)$ being a self-similar group in the sense of Nekrashevych \cite{Nek05} (except that the action on $E^{1*}$ may not be faithful). For this reason, we call $(G, \widetilde E, \widetilde \p)$ the {\em induced self-similar group} of $(G, E, \p)$. We note that any self-similar graph action $(G, E, \p)$ such that $E$ is finite and has only one vertex will satisfy (SS1)--(SS8), and so from now on we will call any such triple a {\em self-similar group}.

We say that $(G, E, \p)$ is {\em pseudo-free}\footnote{This was originally termed {\em residually free} in a preprint of \cite{EP13}.} if whenever we have $g\in G$ and $e \in E^1$ such that $ge = e$ and $\p(g, e) = 1$ then we have that $g = 1$. This is equivalent to saying that whenever we have $g\in G$ and $w \in E^*$ such that $gw = w$ and $\p(g, w) = 1$ then we have that $g = 1$ (\cite{EP13}, Proposition 5.2). Because this property is phrased in terms of the action and cocycle on the edge set, it is clear that pseudo-freeness of $(G, E, \p)$ is equivalent to pseudo-freeness of $(G, \widetilde E, \widetilde\p)$.

Given $(G, E, \p)$, we let (as in \cite{EP13}),
\[
\s_{G, E} = \{ (\alpha, g, \beta)\in E^*\times G \times E^*\mid d(\alpha) = gd(\beta)\}.
\]
This set becomes an inverse semigroup when given the operation
\[
(\alpha, g, \beta)(\gamma, h, \nu) = \begin{cases}(\alpha g\gamma', \p(g, \gamma')h, \nu), &\text{if }\gamma = \beta\gamma',\\ (\alpha, g\p(h^{-1}, \beta')^{-1}, \nu h^{-1}\beta'), & \text{if } \beta = \gamma\beta',\\ 0 &\text{otherwise}\end{cases}
\]
with
\[
(\alpha, g, \beta)^* = (\beta, g^{-1}, \alpha).
\]
Recall that an inverse semigroup with zero $\s$ is called {\em E*-unitary} if whenever one has $s\in \s$ and $e\in E(\s)$,
then $se\in E(\s)\setminus\{0\}$ implies that $s\in E(\s)$.  In \cite{EP13}, Proposition 5.4, it is shown that $\s_{G, E}$ is E*-unitary if and only if $(G, E, \p)$ is pseudo-free. In the remainder of this section, we show that pseudo-freeness is in fact equivalent to a stronger condition on $\s_{G, E}$.

A {\em prehomomorphism} from an inverse semigroup with zero $\s$ to a group $H$ is a function
\[
\theta: \s\setminus \{0\} \to H
\]
such that whenever we have $s, t\in \s\setminus \{0\}$ such that $st\neq 0$, then $\theta(st) = \theta(s)\theta(t)$. A prehomomorphism $\theta$ defined on $S$ is called {\em idempotent pure} if $\theta^{-1}(1) = E(\s)$. Every inverse semigroup $\s$ admits a prehomomorphism into a group $U(\s)$ called the {\em universal group} of $\s$. The group $U(\s)$ is generated by the set $\s\setminus \{0\}$ subject to the relations $s \cdot t = st$ if $st \neq 0$. \ An inverse semigroup $\s$ is called {\em strongly E*-unitary} if there exists an idempotent pure prehomomorphism from $\s$ to a group $H$. This is equivalent to saying that the natural map $\sigma$ from $\s\setminus\{0\}$ to $U(\s)$ (which is a prehomomorphism) is idempotent pure. It is clear that if $\s$ is strongly E*-unitary then it is E*-unitary, because if $se$ is a nonzero idempotent then $1 = \sigma(se) = \sigma(s)$.

In the special case of a self-similar group $(G, E, \p)$, $\s_{G, E}$ has extra structure. In this case, the set of elements of the form $(\alpha, g, \varnothing) \in \s_{G, E}$ form a subsemigroup of $\s_{G,E}$. This semigroup is isomorphic to what is called the {\em Zappa-Sz\'ep product} of the free semigroup $E^*$ by the group $G$, denoted $E^*\bowtie G$. See \cite{La08}, Section 3 for a discussion of the construction of this semigroup. As a set, $E^*\bowtie G$ is $E^*\times G$ and the semigroup operation is
\[
(\alpha, g)(\beta, h) = (\alpha g\beta, \p(g, \beta)h).
\]
One sees that this agrees with the operation from $\s_{G,E}$ restricted to elements of the form $(\alpha, g, \varnothing)$.

\begin{lem}\label{SEUC}(Lawson-Wallis)
Let $(G, E, \p)$ be a self-similar group. Then the inverse semigroup with zero $\s_{G, E}$ is strongly E*-unitary if and only if $E^*\bowtie G$ is cancellative.
\end{lem}
\begin{proof}
By \cite{La08}, $E^*\bowtie G$ is a left Rees monoid. By \cite{LW13} Theorem 5.5, a left Rees monoid can be embedded into a group if and only if it is cancellative. By \cite{La99} Theorem 8, $E^*\bowtie G$ can be embedded into a group if and only if $\s_{G, E}$ is strongly E*-unitary. The result follows.
\end{proof}

\begin{lem}\label{CRF}
Let $(G, E, \p)$ be a self-similar group. Then the semigroup $E^*\bowtie G$ is cancellative if and only if $(G, E, \p)$ is pseudo-free.
\end{lem}
\begin{proof}
We first prove the ``only if'' part. Suppose that $E^*\bowtie G$ is cancellative, and suppose that we have $g\in G$ and $e\in E^1$ such that $ge = e$ and $\p (g, e) = 1$. Then we calculate
\[
(\varnothing, g)(e, 1) = (\varnothing ge, \p(g, e)1) = (e, 1)
\]
\[
(\varnothing, e)(e, 1) = (e, 1).
\]
Since we assume cancellation, this implies that $(\varnothing, g) = (\varnothing, 1)$, and so $g = 1$. Hence $(G, E, \p)$ is pseudo-free.

We now prove that ``if'' part. Suppose that $(G, E, \p)$ is pseudo-free. It is straightforward to show that $E^*\bowtie G$ is always left cancellative. Suppose then that we have $v, v', w\in X^*$ and $g, g', h\in G$ such that
\[
(v, h)(w, g) = (v', h')(w, g).
\]
This implies
\[
(v(hw), \p(h, w)g) = (v'(h'w), \p(h', w) g).
\]
Equating the second coordinates gives us that $\p(h, w) = \p(h',w)$. Also, since the action of $G$ preserves length, equating the first coordinates implies that $v = v'$ and $hw = h'w$. By properties of $\p$ we have
\begin{eqnarray*}
\p(h^{-1}h', w) &=& \p(h^{-1}, h'w)\p(h', w)\\
& = & \p(h^{-1}, hw)\p(h,w)\\
& = & (\p(h,w))^{-1}\p(h,w)\\
& = & 1.
\end{eqnarray*}
Also, by the above we have $h^{-1}h'w = w$. Since $(G, E, \p)$ is pseudo-free, this implies that $h = h'$. Thus we have proven that $(v, h) = (v', h')$ and so $E^*\bowtie G$ is cancellative.
\end{proof}
We now prove the main result of this section, which addresses when $\s_{G, E}$ is strongly E*-unitary for an arbitrary self-similar graph action $(G, E, \p)$ using the two previous lemmas and the induced self-similar group $(G, \widetilde E, \widetilde \p)$.
\begin{theo}\label{RFstronglyE*unitary}
Let $(G, E, \p)$ be a self-similar graph action. Then the inverse semigroup $\s_{G, E}$ is strongly E*-unitary if and only if $(G, E, \p)$ is pseudo-free.
\end{theo}
\begin{proof}
We prove the ``if'' part first. Suppose that $(G, E, \p)$ is pseudo-free (and therefore so is $(G, \widetilde E, \widetilde\p)$). Hence by Lemmas \ref{SEUC} and \ref{CRF}, $\s_{G, \widetilde E}$ is strongly E*-unitary. Define $Q: E^* \to \widetilde E^*$ by
\begin{equation}\label{Qdef}
Q(\alpha) = \begin{cases}\alpha & \text{if }\alpha\notin E^0\\ \varnothing &\text{if }\alpha\in E^0
\end{cases}
\end{equation}
and define a function $\iota: \s_{G, E} \to \s_{G, \widetilde E}$ by
\[
\iota(\alpha, g, \beta) = (Q(\alpha), g, Q(\beta)).
\]
Then $\iota(st) = \iota(s)\iota(t)$ as long as $st \neq 0$. Then if $\sigma: \s_{G, \widetilde E} \to U(\s_{G, \widetilde E})$ is the standard map from $\s_{G, \widetilde E}$ to its universal group, the map
\[
\sigma\circ\iota: \s_{G, E} \to U(\s_{G, \widetilde E})
\]
is a prehomomorphism. Since $\s_{G, \widetilde E}$ is strongly E*-unitary, $\sigma$ is idempotent pure. Take $(\alpha, g, \beta)\in \s_{G, E}$. If $\sigma\circ\iota(\alpha, g, \beta) = 1$, then this implies that $(Q(\alpha), g, Q(\beta))$ is a nonzero idempotent, that is, $Q(\alpha) = Q(\beta)$ and $g = 1$. Hence either $\alpha = \beta$ or both $\alpha$ and $\beta$ are paths of length zero (ie vertices). It is not possible for $\alpha$ and $\beta$ to be different paths of length zero, because $(\alpha, g, \beta)\in \s_{G, E}$ and $g = 1$ implies that $d(\alpha) = d(\beta)$. Hence $\sigma\circ\iota$ is an idempotent pure prehomomorphism and thus $\s_{G, E}$ is strongly E*-unitary.

Now we prove the ``only if'' part. Suppose that $\s_{G, E}$ is strongly E*-unitary. Then, in particular, it is E*-unitary. By \cite{EP13}, Proposition 5.4, this implies that $(G, E, \p)$ is pseudo-free.
\end{proof}
\begin{ex} \label{odo1}{\bf (The Odometer)} Take a natural number $n\geq 2$ and consider the graph with $n$ edges and one vertex $R_n$. That is, $R_n^1= \{0,1, 2, \dots n-1\}$, and $R_n^1 = \{\varnothing\}$. We write the group $\bz$ of integers multiplicatively, with generator $z$, so that $\bz = \{z^m\mid m\in \bz\}$. For $x\in R_n^1$, let
\[
zx = \begin{cases} x+1 & \text{ if }x \neq n - 1\\ 0 & \text{ if }x = n-1 \end{cases}.
\]
This formula defines a self-similar group $(\bz, R_n, \p)$ with cocycle
\[
\p(z, x) = \begin{cases}e & \text{ if }x \neq n-1 \\ z &\text{ if } x = n-1\end{cases}.
\]
We claim that $(\bz, R_n, \p)$ is pseudo-free. Suppose that $\nu\in R_n^*$, $z^m\in \bz$, $z^m\nu = \nu$ and $\p(z^m, \nu) = 1$. We suppose that $m > 0$. Since $\p(z^m, \nu) = 1$, we must have that $|\nu| > \log_n(m)$. Furthermore, since $z^m \nu = \nu$, we must have that $|\nu|$ is a multiple of $\log_n(m)$, which is impossible. The case of $m < 0$ is similar. Hence, we must have that $m =0$, and so $(\bz, R_n, \p)$  is pseudo-free.

By Theorem \ref{RFstronglyE*unitary}, $\s_{\bz, R_n}$ is strongly E*-unitary, and so the identity map $\sigma: \s_{\bz, R_n}\setminus \{0\}\to U(\s_{\bz, R_n})$ is idempotent pure. We describe the universal group $U(\s_{\bz, R_n})$. One can always assume that the universal group is generated by the image of $\sigma$. Recall that
\[
\s_{\bz, R_n} = \{(\alpha, z^n, \beta)\mid n\in\bz, \alpha, \beta\in R_n^*\}.
\]
Let
\[
a_i:= \sigma((i, z^0, \varnothing)), \hspace{0.2cm}i\in R_n^1, \hspace{1cm} Z:= \sigma((\varnothing, z^1, \varnothing)).
\]
Then, using the relations in $\s_{\bz, R_n}$ we find quickly that
\[
U(\s_{\bz, R_n}) = \left\langle a_0, a_1, \dots a_{n-1}, Z\mid Za_i = a_{i+1}\text{ for } 0\leq i < n-1, Za_{n-1} = a_0Z\right\rangle.
\]
We can use the relations to reduce this to
\[
U(\s_{\bz, R_n}) = \left\langle a_0, Z\mid Z = a_0^{-1}Z^na_0\right\rangle
\]
obtaining that $U(\s_{\bz, R_n})$ is isomorphic to the Baumslag-Solitar group $BS(1, n)$. We note that the relationship between this self-similar action and these groups has been noted in \cite{BRRW14}, Example 3.5. There it is seen directly that the Zappa-Sz\'ep product $R_n^*\bowtie \bz$ embeds as a suitable positive cone in $BS(1,n)$.
\end{ex}

\section{The C*-algebra of $(G, E, \p)$ as a Partial Crossed Product}

To a self-similar graph action $(G, E, \p)$ one associates a C*-algebra, denoted in \cite{EP13} as $\mathcal{O}_{G, E}$.

\begin{defn}
Let $(G, E, \p)$ be a self-similar graph action. Then $\mathcal{O}_{G, E}$ is the universal C*-algebra for the set
\[
\{p_x\mid x\in E^0\}\cup\{s_e\mid e\in E^1\}\cup\{u_g\mid g\in G\}
\]
subject to the following:
\begin{enumerate}\addtolength{\itemsep}{-0.5\baselineskip}
\item[(CK1)] $\{p_x\mid x\in E^0\}$ is a set of mutually orthogonal projections,
\item[(CK2)] $\{s_e\mid e\in E^1\}$ is a set of partial isometries,
\item[(CK3)] $s_e^*s_e = p_{d(e)}$ for each $e\in E^1$,
\item[(CK4)] $p_x = {\displaystyle \sum_{e\in r^{-1}(x)}s_es_e^*}$ for each $x\in E^0$ with $0 < r^{-1}(x) < \infty$,
\item[(EP1)] $g\mapsto u_g$ is a unitary representation of $G$,
\item[(EP2)] $u_gs_e = s_{ge}u_{\p(g,e)}$ for all $g\in G$ and $e\in E^1$,
\item[(EP3)] $u_gp_x = p_{gx}u_g$ for all $g\in G$ and $x\in E^0$.
\end{enumerate}
\end{defn}
The first four relations are the Cuntz-Krieger relations for the graph $E$. Our main theorem describes when this C*-algebra is isomorphic to a partial crossed product, and is a combination of Theorem \ref{RFstronglyE*unitary} with the following previously known results.

\begin{theo}(\cite{EP13}, Theorem 9.5)\label{EPTheo}
Let $(G, E, \p)$ be a self-similar graph action such that $E$ is a finite graph with no sources. Then $\mathcal{O}_{G, E}$ is isomorphic to the C*-algebra of the groupoid $\Gt(\s_{G, E})$.
\end{theo}

\begin{theo}(\cite{SM11}, Theorem 5.3)\label{SMTheo}
Let $\s$ be a countable E*-unitary inverse semigroup with universal group $U(\s)$ and tight spectrum $\Et$. Then there is a natural partial action of $U(\s)$ on $\Et$ such that the groupoids $\Gt(\s)$ and $U(\s)\ltimes \Et$ are isomorphic.
\end{theo}

We now define what is meant by ``$\Et$'' and ``$\Gt(\s)$'' in the above, and then describe the situation in our case.

Each inverse semigroup $\s$ possesses a natural order structure. Two elements $s$ and $t$ satisfy $s \leqslant t$ if and only if $s = ss^*t$. This is equivalent to saying that $s = te$ for some idempotent $e$. If $e$ and $f$ are idempotents, then $e\leqslant f$ if and only if $ef = e$. Recall that a {\em filter} $F$ in a partially ordered set $X$ is a {\em proper} subset that is {\em downward directed} (that is, for each $x,y\in F$ there is an element $z\in F$ such that $z \leqslant x, y$) and {\em upwards closed} (that is, if $x\in F$ and $x \leqslant y$ then $y\in F$). If $F$ is a proper subset which is downward directed, then it is called a {\em filter base} and 
\[
\overline{F} = \{ x\in X \mid f \leqslant x \text{ for some } f\in F\}
\]
is a filter. Also recall that an {\em ultrafilter} is a filter which is not properly contained in another filter. Filters in $E(\s)$ are closed under multiplication and, if $\s$ has a zero element, then filters in $E(S)$ do not contain the zero element. If $\xi\subset E(\s)$ is a filter and $e\in \xi$, then it is straightforward that both $e\xi$ and $\xi e$ are filter bases and $\xi = \overline{e\xi} = \overline{\xi e}$. 

Suppose that $\s$ is countable and consider $\{0,1\}^{E(\s)}$, the power set of $E(\s)$. This is a compact Hausdorff space homeomorphic to the Cantor set when given the product topology. Let $\widehat{E}_0$ denote the closed subspace of filters in $E(\s)$ -- this is called the {\em spectrum} of $\s$. Let $\Eu$ denote the space of ultrafilters, and let $\Et$ denote the closure of $\Eu$ in $\widehat{E}_0$ -- this is called the {\em tight spectrum} of $\s$.

Any inverse semigroup acts naturally on its spectra. Fix an inverse semigroup $\s$ with set of idempotents $E$. For each $e\in E$, let $D_e = \{ \xi \in \Et\mid e\in \xi\}$. Then define $\theta_s: D_{s^*s}\to D_{ss^*}$ by
\[
\theta_s(\xi) = \overline{s\xi s^*}.
\]
These sets and maps define an action of $\s$ on $\Et$. The groupoid of germs associated to this action is called the {\em tight groupoid} of $\s$ and is denoted $\Gt(\s)$. For details, see \cite{Ex08}.

We now turn to constructing a partial action of $U(\s)$ on the tight spectrum of $\s$ from the canonical action of $\s$. Let $\sigma: \s\setminus \{0\}\to U(\s)$ be the standard identity map from $\s\setminus \{0\}$ to $U(\s)$. The basic idea from \cite{SM11} is that for a given group element $g\in U(\s)$ one ``bundles together'' the partial homeomoprhisms corresponding to all its preimages under $\sigma$. It turns out that to guarantee that such functions agree on their domains, one needs that $\s$ is strongly E*-unitary.

To be more precise, let $(\{D_e\}_{e\in E(\s)}, \{\theta_s\}_{s\in \s})$ be the canonical action of $\s$ on $\Et$. Define a partial action $(\{F_g\}_{g\in U(\s)}, \{\widetilde\theta_g\}_{g\in U(\s)}\})$ of the group $U(\s)$ on $\Et$, by setting
\begin{equation}\label{PADomain}
F_g = \bigcup_{s\in \sigma^{-1}(g)}D_{ss^*}
\end{equation}
such that, for all $\xi$ in some $D_{s^*s}\subset F_{g^{-1}}$, we have $\widetilde\theta_g(\xi) = \theta_s(\xi)$. To see why this is well-defined, suppose that we have $s, t\in \sigma^{-1}(g)$ for some $g\in U(\s)$, and that we have a filter $\xi\in D_{s^*s}\cap D_{t^*t}$ -- this means that $s^*s, t^*t \in \xi$, and so $s^*st^*t \neq 0$. Hence $st^* \neq 0$. A short calculation gives $$\sigma(st^*) = \sigma(s)\sigma(t)^{-1} = gg^{-1} = 1,$$ and so $st^* = ts^*$ is an idempotent.

Since $t^*t \in \xi$ and $s\xi s^*$ is a filter, we must have that $st^*ts^*\neq 0$. This means that $st^*st^* \neq 0$,
and so similar to the above $t^*s \neq 0$ and is an idempotent. Furthermore, $s^*t = s^*ts^*t = s^*st^*t \in \xi$. Now, suppose that $e\in \overline{s\xi s^*}$. Then there exists $f\in \xi$ such that $sfs^* \leqslant e$, which is to say that $esfs^* = sfs^*$. Now we have 
\[
e(ts^*sfs^*st^*) = ts^*(esfs^*)st^* = ts^*sfs^*st^* \in t\xi t^*
\]
and so $e$ is greater than an element of $t\xi t^*$, whence $e\in \overline{t\xi t^*}$. This argument is symmetric in $s$ and $t$, so we must have that $\overline{s\xi s^*}=\overline{t\xi t^*}$, and thus the functions $\theta_s$ and $\theta_t$ agree on $F_{g^{-1}}$ from \eqref{PADomain}, and so $\widetilde\theta_g$ is well-defined. We note that by \eqref{PADomain}, the only $g\in U(\s)$ for which $F_g$ is nonempty will be those in the image of $\sigma$.

It is straightforward to show that map from $\mathcal{G}(\s, \Et, \theta)$ to $U(\s)\ltimes_{\widetilde\theta} \Et$ given by
\[
[s, \xi] \mapsto (\sigma(s), \xi)
\]
is a well-defined isomorphism of topological groupoids.

Now, let $(G, E, \p)$ be a self-similar graph action such that $E$ is a finite graph with no sinks or sources. Then the tight spectrum of $\s_{G, E}$ is homeomorphic to the space of infinite paths, see \cite{EP13}, Section 8. The action of $G$ on $E^*$ extends to an action on $\Sigma_E$ by homeomorphisms determined by the following: for each $g\in G$ and $x\in \Sigma_E$ we have
\[
(gx)_i = \p(g, x_1x_2\cdots x_{i-1})x_i.
\]
As above, each $(\alpha, g, \beta)\in \s_{G, E}$ acts via a partial homeomorphism on its tight spectrum $\Sigma_E$ given by
\[
(\alpha, g, \beta): \beta\Sigma_E \to \alpha\Sigma_E
\]
\[
\beta x\mapsto \alpha (gx).
\]

The above together with Theorem \ref{RFstronglyE*unitary} and previously known Theorems \ref{EPTheo} and \ref{SMTheo} directly imply the following.
\begin{theo}\label{maintheo}
Let $(G, E, \p)$ be a self-similar graph action such that $E$ is a finite graph with no sinks or sources. Suppose further that $(G, E, \p)$ is pseudo-free. Then $\mathcal{O}_{G, E}$ is isomorphic to a partial crossed product $C(\Sigma_E)\rtimes_{\widetilde\theta} U(\s_{G, E})$. The action $\widetilde\theta$ is the action on $C(\Sigma_E)$ derived from the action $\theta$ described above.
\end{theo}

\begin{ex} \label{odo2}{\bf (The 2-Odometer)} We consider the self-similar group $(\bz,R_2, \p)$, that is, Example \ref{odo1} with $n=2$. Recall that
\[
\s_{\bz,R_2} = \{(\alpha, z^n, \beta)\mid n\in\bz, \alpha, \beta\in R_2^*\}.
\]
For convenience and clarity of computations, we will make the identifications $S_\alpha :=(\alpha, z^0, \varnothing)$ and $U :=(\varnothing, 1, \varnothing)$, so that
\[
\s_{\bz,R_2} = \{S_\alpha U^n S_\beta^*\mid n\in\bz, \alpha, \beta\in \{0,1\}^*\}.
\]
The universal group of $\s_{\bz, R_2}$ is the Baumslag-Solitar group $BS(1,2)$. We give a different presentation than in Example \ref{odo1} by setting
\begin{eqnarray*}
\sigma(S_0) &=& a\\
\sigma(S_1) &=& b.
\end{eqnarray*}
Doing this, we obtain
\[
H := U(\s_{\bz, R_2}) = \left\langle a,b \mid aba^{-1}b^{-1} = ba^{-1}\right\rangle
\]
with $\sigma(U) = ba^{-1}$.

Because this action is pseudo-free, $\sigma$ is idempotent-pure and Theorem \ref{SMTheo} applies. Hence the C*-algebra of the tight groupoid of $\s_{\bz,R_2}$ is isomorphic to the partial crossed product of the tight spectrum of $\s_{\bz, R_2}$ by $H$. We will describe this partial action $(\{D_g\}_{g\in H}, \{\theta_g\}_{g\in H})$.

We begin by looking closer at the group $H$.
For a word $\alpha$ in $a$ and $b$, let $\widetilde\alpha$ be the word in $0$ and $1$ obtained from $\alpha$ by replacing $a$ with $0$ and $b$ with $1$; that is, $\alpha = \sigma(S_{\widetilde\alpha})$. Supppse that $\alpha$ and $\beta$ are words in $a$ and $b$ and that $|\alpha| = |\beta|$. Then we claim that the group element $\alpha\beta^{-1}$ is equal to $\sigma(U^k)$ for some $k\in\bz$. One may view $\widetilde\alpha$ and $\widetilde\beta$ as binary numbers (with the powers of 2 increasing from left to right rather than right to left). For a word $\nu$ in $a$ and $b$, let $n_\nu$ denote the integer corresponding to the binary number determined by $\widetilde\nu$. Now, notice that
\begin{eqnarray*}
U^{n_\beta - n_\alpha}S_{\widetilde\alpha} S_{\widetilde\beta}^* &=& U^{n_\beta}U^{- n_\alpha}S_{\widetilde\alpha} S_{\widetilde\beta}^*\\
 &=& U^{n_\beta}S_{0^{|\alpha|}}S_{\widetilde\beta}^*\\
 &=& S_{\widetilde\beta} S_{\widetilde\beta}^* \in E(\s_{\bz, R_2})
\end{eqnarray*}
where, in the above, $0^{|\alpha|}$ denotes the word consisting of $|\alpha|$ 0's. Thus, $\sigma(U^{n_\beta - n_\alpha}S_{\widetilde\alpha} S_{\widetilde\beta}^*) = 1$, and so $\alpha\beta^{-1} = \sigma(U^{n_\alpha - n_\beta})$ as
desired. In fact, from this it is easy to see that $\sigma(U^k)\in H$ can always be written in the form
$\alpha\beta^{-1}$ for words $\alpha, \beta$ in $a$ and $b$ with $|\alpha| = |\beta|$. Hence, the image of $\sigma$ consists only of group elements in $H$ which can be written in the form $\alpha\beta^{-1}$ for
words $\alpha$ and $\beta$ (not necessarily of the same length). Since $D_g$ is only nonempty if $g$ is in the image of $\sigma$, the only group elements for which $D_g$ will be nonempty are those which can be written in the form $\alpha\beta^{-1}$.

Next we further clarify the different forms that group elements can take in $H$.

\vspace{0.2cm}
\noindent {\bf Claim:} If $\alpha\beta^{-1} = \nu\omega^{-1}$ in $H$, then $|\alpha| - |\beta| = |\nu| - |\omega|$. If
$|\alpha|- |\beta| >0$, then the initial segment of $\alpha$ of length $|\alpha|- |\beta|$ is equal to the initial segment of $\nu$. Similarly, if $|\alpha|-|\beta|<0$, the initial segment of $\beta$ of length $|\beta|-|\alpha|$ is equal to the initial segment of $\omega$.

\noindent {\em Proof.} First suppose that $|\alpha| = |\beta|$. Then by the discussion above, $\alpha\beta^{-1} = \sigma(U^n)$ for some $n$. Hence we must have that
\[
\sigma(U^nS_{\widetilde\omega}S_{\widetilde\nu}^*) = \alpha\beta^{-1}\omega\nu^{-1} = 1.
\]
Since $\sigma$ is idempotent pure, $U^nS_{\widetilde\omega}S_{\widetilde\nu}^*$ must be an idempotent, and so $|\omega| = |\nu|$.

Now suppose that $|\alpha| > |\beta|$, Then $\alpha\beta^{-1} = \alpha_0\alpha_1\beta^{-1}$ with $|\alpha_0| = |\beta|$. Again, by the discussion above, $\alpha_1\beta^{-1} = \sigma(U^n)$  for some $n$. Hence
\[
\sigma(S_{\widetilde\alpha_0}U^nS_{\widetilde\omega} S_{\widetilde\nu}^*) = \alpha_0\alpha_1\beta^{-1}\omega\nu^{-1} = 1
\]
and again $S_{\widetilde\alpha_0}U^nS_{\widetilde\omega} S_{\widetilde\nu}^*$ is an idempotent. This can only happen if $\alpha_0$ is an initial segment of $\nu$ and $|\alpha_0| + |\omega| = |\nu|$. Hence $|\alpha| - |\beta| = |\nu| - |\omega|$. The case of $|\beta| > |\alpha|$ is similar. \hspace{\fill}$\square$

\vspace{0.2cm}
As above, the tight spectrum of $\s_{\bz, \{0,1\}}$ is homeomorphic to $\Sigma_{\{0,1\}}$, the space of infinite words in $0$ and $1$. There is a homeomorphism $\lambda: \Sigma_{\{0,1\}} \to \Sigma_{\{0,1\}}$ which takes an infinite sequence $x$, looks for the first entry which is not equal to 1, switches it to 1, switches the previous entries to 0, and leaves the rest of the entries unchanged. If all entries are equal to 1, $\lambda$ switches them all to 0. One sees that this is the extension of the action of $\bz$ in Example \ref{odo1} to infinite sequences. Our maps $\theta_g$ will involve this homeomorphism.

We can now describe the partial action $(\{D_g\}_{g\in H}, \{\theta_g\}_{g\in H})$. If $g$ is not of the form $\alpha\beta^{-1}$ for some words $\alpha, \beta$ in $a$ and $b$, then $D_g = \emptyset$. Otherwise, we have three cases.

\begin{enumerate}
\item If $g = \alpha\beta^{-1}$ with $|\alpha| = |\beta|$, then $D_g = D_{g^{-1}} = \Sigma_{\{0,1\}}$, and
\[
\theta_g(x) = \lambda^{n_\alpha - n_\beta}(x).
\]
\item If $g = \alpha\beta^{-1}$ with $|\alpha| > |\beta|$, then $\alpha\beta^{-1} = \alpha_g\alpha_0\beta^{-1}$ with $|\alpha_0| = |\beta|$. Further, the word $\alpha_g$ does not depend on the particular representation of $\alpha\beta^{-1}$ by the above claim. In this case we have $D_{g^{-1}} = \Sigma_{\{0,1\}}$ and $D_g = \alpha_g\Sigma_{\{0,1\}}$. The map $\theta_g$ is given by
\[
\theta_g(x) = \alpha_g\lambda^{n_{\alpha_0}-n_\beta}(x)
\]
where above we are concatenating the infinite sequence $\lambda^{n_{\alpha_0}-n_\beta}(x)$ with $\alpha_g$.
\item The third case, where $g= \alpha\beta^{-1}$ with $|\alpha| > |\beta|$, is completely determined by the second case above. Here $\alpha\beta^{-1} = \alpha\beta_0^{-1}\beta_g^{-1}$ with $|\alpha| = |\beta_0|$, $\theta_g: \beta_g\Sigma_{\{0,1\}} \to \Sigma_{\{0,1\}}$, and
\[
\theta_g(\beta_gx) = \lambda^{n_{\alpha}-n_{\beta_0}}(x)
\]
\end{enumerate}
The partial action $\theta$ induces a partial action $\widetilde\theta$ on $C(\Sigma_E)$. By Theorem \ref{maintheo} and Example \ref{odo1}, we have that
\[
\mathcal{O}_{\bz, R_2} \cong C(\Sigma_{R_2})\rtimes_{\widetilde\theta} BS(1, 2)
\]
and, more generally
\[
\mathcal{O}_{\bz, R_n} \cong C(\Sigma_{R_n})\rtimes_{\widetilde\theta} BS(1, n)
\]
We finish this example by pointing out that this gives a realization of the C*-algebra $\mathcal{Q}_2$ associated in \cite{LL12} to the $2$-adic integers as a partial crossed product, because  $\mathcal{Q}_2\cong \mathcal{O}_{\bz, R_2}$. See \cite{BRRW14}, Example 6.5 for more details.
\end{ex}

\section{Induced Actions and $\mathcal{O}_n$}

Suppose that $\theta = (A, G, \{D_g\}_{g\in G}, \{\theta_g\}_{g\in G})$ is a partial C*-dynamical system and that $G$ is a subgroup of $H$. Then one can extend this partial C*-dynamical system to $H$, creating $\widehat\theta = (A, H, \{\widehat D_g\}_{h\in H}, \{\widehat\theta_g\}_{h\in H})$ by setting $\widehat D_h = D_h$ and $\widehat \theta_h = \theta_h$ if $h\in G$ and $\widehat D_h$ and $\widehat \theta_h$ to be the zero ideal and zero map if $h\notin G$. In this situation, $A\rtimes_\theta G \cong A\rtimes_{\widehat\theta} H$. In this way we see that a result of the form ``$B$ is isomorphic to a partial crossed product by $H$'' doesn't contain as much information as one would like, because perhaps a subgroup of $H$ would suffice. In our cases so far, we have shown that the $\mathcal{O}_{G, E}$ are isomorphic to crossed products by certain groups, and that these groups are actually generated by the elements whose corresponding ideals are nonempty.

It is well-known that the Cuntz algebra $\mathcal{O}_n$ can be realized as a partial crossed product. The first such construction appears in \cite{QR97}, where it is shown that $\mathcal{O}_n$ is isomorphic to a partial crossed product of the Cantor set by $\F_n$, the free group on $n$ elements. We reproduce this construction below in Example \ref{cuntzex}. In \cite{Ho07}, $\mathcal{O}_n$ is realized as a partial crossed product of the Cantor set by the Baumslag-Solitar group $BS(1, n) \cong \bz\left[\frac1n\right]\rtimes \bz$, an amenable group. In both of these situations, the elements $g$ for which the ideal $D_g$ is nonzero generate the group. In this section, we show that if $(G, E, \p)$ is pseudo-free self-similar graph action such that $E$ has only one vertex and which satisfies a condition we call {\em exhausting} (Definition \ref{exhausting}), then $\mathcal{O}_{|E^1|}$ is isomorphic to a partial crossed product by the group $U(\s_{G, E})$ and the group elements $g$ such that the ideal corresponding to $g$ is nonzero generate $U(\s_{G, E})$.

\begin{ex} {\bf (The Cuntz Algebra)} \label{cuntzex} This is a construction seen in \cite{QR97}. Let $A$ be a finite alphabet and consider the space $\Sigma_A$ of right-infinite words in elements of $A$. Given the product topology, $\Sigma_A$ is homeomorphic to the Cantor set. Let $\F_A$ denote the free group generated by $A$. We will describe a partial action of $\F_A$ on $\Sigma_A$. If $g\in \F_A$ is not of the form $\alpha\beta^{-1}$ for words $\alpha, \beta\in A^*$, then $D_g = \emptyset$. Otherwise, for all $\alpha, \beta\in A^*$ we have
\[
D_{\alpha\beta^{-1}} = \alpha\Sigma_A = \{\alpha x\mid x\in \Sigma_A\}
\]
and the map $\theta_{\alpha\beta^{-1}}:\beta\Sigma_A \to \alpha\Sigma_A$ is defined by
\[
\theta_{\alpha\beta^{-1}}(\beta x) = \alpha x.
\]
Here, $\Sigma_A$ is compact and each $D_g$ is clopen. One can show that the C*-algebra of the groupoid $\F_A \ltimes_\theta \Sigma_A$ is isomorphic to $\mathcal{O}_{|A|}$.

\end{ex}

\begin{defn}Suppose that we have a partial action of a group $G$ on a set $X$, say $(\{D_g\}_{g\in G}, \{\theta_g\}_{g\in G})$ and suppose that $\varphi: G \to H$ is an onto homomorphism of groups. Suppose further that whenever $h\in H$,  $g_1, g_2\in \varphi^{-1}(h)$, and $x\in D_{g_1^{-1}}\cap D_{g_2^{-1}}$, then $\theta_{g_1}(x) = \theta_{g_2}(x)$. Then the {\em induced partial action} of $H$ on $X$ is the pair
\[
(\{E_h\}_{h\in H}, \{\theta_h\}_{h\in H})
\]
where
\[
E_h = \bigcup_{g\in \varphi^{-1}(h)} D_g
\]
and, with slight abuse of notation, the $\theta_h$ are as before. 
\end{defn}
This is well-defined because these functions agree on any possible intersections of the sets above. We will need the following condition.


\begin{defn}\label{exhausting}
A self-similar graph action $(G, E, \p)$ is called {\em exhausting} if for all $g\in G$ there exists $\alpha\in E^*$ such that $\p(g, \alpha) = 1_G$.
\end{defn}
There is a natural homomorphism $\phi:\F_{E^1} \to U(\s_{G, E})$ determined by $\phi(e) = \sigma(S_e)$ for all $e\in E^1$. In the case that $(G, E, \p)$ is exhausting, we have the following.
\begin{lem}\label{exhaustingsurjective}
Let $(G, E, \p)$ be a self-similar graph action such that $E$ is a finite graph with no sinks or sources. If $(G, E, \p)$ is exhausting, then the natural group homomorphism $\phi:\F_{E^1}\to U(\s_{G, E})$ is surjective.
\end{lem}
\begin{proof}
Let $g\in G$ and find $\alpha\in E^*$ such that $\p(g,\alpha) = 1_G$. Let $\sigma: \s_{G, E}\setminus \{0\}\to
U(\s_{G, E})$ be the natural map.  Then $U_gS_\alpha S^*_{g\alpha} = S_{g\alpha}S^*_{g\alpha}$, and so $\sigma(U_gS_\alpha S^*_{g\alpha}) = 1$. Since $\sigma$ is multiplicative on nonzero products,  we have that $\sigma(U_g) = \sigma(S_{g\alpha}S^*_g)$. We know that $U(\s_{G, E})$ is generated by the image of $\s_{G, E}$, so this implies that $U(\s_{G, E})$ is generated by $\{\sigma(S_x)\}_{x\in E^1}$, and the result follows.
\end{proof}
In the following, we refer to the action in Example \ref{cuntzex} as the {\em Quigg-Raeburn action}.
\begin{theo}\label{inducedcuntz}
Suppose that $(G, E, \p)$ is a self-similar group. Suppose also that $(G, E, \p)$ is pseudo-free and exhausting. Then the partial crossed product associated to the action $\theta$ of $U(\s_{G, E})$ on $\Sigma_E$ induced by the Quigg-Raeburn action is isomorphic to $\mathcal{O}_{|E^1|}$, and the elements $g\in U(\s_{G, E})$ with corresponding ideals not equal to 0 generate $U(\s_{G, E})$.
\end{theo}
\begin{proof}
We first show that the induced action is well-defined. Let $\phi: \F_{E^1} \to U(\s_{G, E})$ denote the surjective group homomorphism from Lemma \ref{exhaustingsurjective}, given on generators by $\phi(x) = \sigma(S_x)$. If we take the $D_{\alpha\beta^{-1}}$ as in Example \ref{cuntzex}, then for $g\in U(\s_{G, E})$ we have
\[
E_h = \bigcup_{\alpha\beta^{-1}\in \phi^{-1}(h)} D_{\alpha\beta^{-1}} = \bigcup_{\alpha\beta^{-1}\in \varphi^{-1}(h)} \alpha\Sigma_E.
\]
In general, given two cylinder sets $\alpha\Sigma_E$ and $\eta\Sigma_E$, either they are disjoint, or one is contained in the other. In the latter case, $\alpha = \eta\alpha'$ (without loss of generality). So, suppose that $h\in U(\s_{G, E})$ and $\alpha\beta^{-1}, \eta\gamma^{-1}\in \phi^{-1}(h)$ such that $\beta\Sigma_E\cap \gamma\Sigma_E \neq \emptyset$. Without loss of generality, we will assume that $\beta\Sigma_E\cap \gamma\Sigma_E = \beta\Sigma_E$, and so $\beta = \gamma\beta'$. This implies that
\begin{eqnarray*}
1 &=& \sigma(S_\alpha S_\beta^*S_\gamma S_\eta^*)\\
  &=& \sigma(S_\alpha S_{\beta'}^*S_\eta^*)\\
&=& \sigma(S_\alpha S_{\eta\beta'}^*)
\end{eqnarray*}
and since $\sigma$ is idempotent pure, we must have that $\alpha = \eta\beta'$. Hence
\[
\alpha\beta^{-1} = \eta\beta' (\gamma\beta')^{-1} = \eta\gamma^{-1},
\]
and so $\theta_{\alpha\beta^{-1}}$ and $\theta_{\eta\gamma^{-1}}$ agree on the intersection of their domains (because if their domains intersect, they must in fact be equal group elements).

Now we have that the induced action is well-defined. The map
\[
\Phi: \F_{E^1} \ltimes_\theta \Sigma_E \to U(\s_{G, E}) \ltimes_\theta \Sigma_E 
\]
\[
\Phi(\alpha\beta^{-1}, \beta x) = (\phi(\alpha\beta^{-1}), \beta x)
\]
is easily shown to be an isomorphism of topological groupoids. We omit the details.

\end{proof}

\begin{ex} If $(\bz, R_n, \p)$ is the odometer from Examples \ref{odo1} and \ref{odo2}, by Theorem \ref{inducedcuntz} we obtain that that $\mathcal{O}_n \cong C(\Sigma_{R_n})\rtimes BS(1, n)$, reproducing the result from \cite{Ho07}.

We also note that while the groups and algebra in this example are the same as Example \ref{odo2}, we obtain nonisomorphic crossed products because the domains of group elements can be different. For example, in the action in Example \ref{odo2} we have that $D_{ba^{-1}}=\Sigma_{R_2}$, while in the induced action described above, 
\[
D_{ba^{-1}} = \bigcup_{n\geq 0} b^na\Sigma_{R_2} = \Sigma_{R_2}\setminus \{bbb\cdots\}.
\]
\end{ex}

\bibliography{C:/Users/Charles/Dropbox/Research/bibtex}{}
\bibliographystyle{alpha}

{\small 
\textsc{Departamento de Matem\'atica, Campus Universit\'ario Trindade
CEP 88.040-900 Florian\'opolis SC, Brasil.}

Charles Starling (corresponding author): \texttt{slearch@gmail.com}}
\end{document}